\documentclass[10pt]{article}
\font\smallit=cmti10
\font\smalltt=cmtt10

\usepackage{amssymb,latexsym,amsmath,epsfig,amsthm} 

\makeatletter

\renewcommand\section{\@startsection {section}{1}{\z@}
{-30pt \@plus -1ex \@minus -.2ex}
{2.3ex \@plus.2ex}
{\normalfont\normalsize\bfseries}}

\renewcommand\subsection{\@startsection{subsection}{2}{\z@}
{-3.25ex\@plus -1ex \@minus -.2ex}
{1.5ex \@plus .2ex}
{\normalfont\normalsize\bfseries}}

\renewcommand{\@seccntformat}[1]{\csname the#1\endcsname. }
\newtheorem{theorem}{Theorem}[section]
\newtheorem{lemma}[theorem]{Lemma}

\theoremstyle{definition}
\newtheorem{definition}[theorem]{Definition}
\theoremstyle{remark}
\newtheorem{remark}[theorem]{Remark}
\newtheorem{qs}[theorem]{Question}

\numberwithin{equation}{section}

\newcommand{\ben}{{\mathbb N}}
\newcommand{\ber}{{\mathbb R}}

\newcommand{\pf}{\mathcal{P}_f}
\newcommand{\emp}{\emptyset}
\newcommand{\nhat}[1]{\{1,2,\discretionary{}{}{}\ldots,\discretionary{}{}{}#1\}}

\newcommand{\vvarphi}{\raise 2 pt \hbox{$\varphi$}}
\newcommand{\cl}{c\ell}

\makeatother

\begin{document}

\begin{center}
\uppercase{\bf ADDITIVE AND MULTIPLICATIVE STRUCTURES OF $C^{\star}$-SETS}
\vskip 20pt
{\bf Dibyendu De\footnote{The author acknowledges the support received from the DST-PURSE programme grant.}}\\
{\smallit Department of Mathematics, University of Kalyani, Kalyani-741235,
West Bengal, India}\\
{\tt dibyendude@klyuniv.ac.in}\\
\vskip 10pt
\end{center}
\vskip 30pt
\centerline{\smallit Received: , Revised: , Accepted: , Published: } 
\vskip 30pt

\centerline{\bf Abstract}

\noindent
It is known that for an IP${^\star}$ set $A$ in $(\mathbb{N},+)$ and  a sequence $\langle x_{n}\rangle_{n=1}^{\infty}$ in $\mathbb{N}$, there exists a sum subsystem $\langle y_{n}\rangle_{n=1}^{\infty}$  of $\langle
x_{n}\rangle_{n=1}^{\infty}$ such that
 $FS(\langle
y_n\rangle_{n=1}^\infty)\cup FP(\langle
y_n\rangle_{n=1}^\infty)\subseteq A$. Similar types of result have also been
proved for central$^{\star}$ sets where the sequences have been considered from the class of minimal sequences. In this present work, we shall prove some analogous results for C$^{\star}$-sets for a more general class of sequences.

\pagestyle{myheadings}
\markright{\smalltt INTEGERS: 13 (2013)\hfill}
\thispagestyle{empty}
\baselineskip=12.875pt
\vskip 30pt

\section{Introduction}

A famous Ramsey theoretic result is Hindman's Theorem :

\begin{theorem}\label{hindman}
Given a finite coloring of $\mathbb{N}=\bigcup_{i=1}^r A_i$,  there exists a sequence $\langle x_{n}\rangle_{n=1}^{\infty}$ in $\mathbb{N}$ and $i\in\{1,2,\ldots ,r\}$ such that
$$FS(\langle x_{n}\rangle_{n=1}^{\infty})=\left\{\sum_{n\in
F}x_n:F\in\mathcal{P}_f(\mathbb{N})\right\}\subseteq A_i,$$
 where for any set $X$,
$\mathcal{P}_f(X)$ is the set of all finite nonempty subsets of $X$.
\end{theorem}

A strongly negative answer to a combined additive and multiplicative version of Hindman's Theorem was presented in ~\cite[Theorem 2.11]{refHpp}. Given a sequence  $\langle
x_{n}\rangle_{n=1}^{\infty}$ in $\mathbb{N}$, let us denote $PS(\langle
x_{n}\rangle_{n=1}^{\infty})$ = $\{x_m+x_n:m,n\in\mathbb{N}$ and
$m\neq n\}$ and $PP(\langle
x_{n}\rangle_{n=1}^{\infty})=\{x_m\cdot x_n:m,n\in\mathbb{N}$
and $m\neq n\}$.

\begin{theorem}\label{notramseyn}

There exists a finite partition
$\mathcal{R}$ of $\mathbb{N}$ with no one-to-one sequence $\langle x_{n}\rangle_{n=1}^{\infty}$ in $\mathbb{N}$
such that $PS(\langle x_{n}\rangle_{n=1}^{\infty})\cup PP(\langle
x_{n}\rangle_{n=1}^{\infty})$ is contained in one cell of the
partition $\mathcal{R}$.
\end{theorem}

The original proof of Theorem \ref{hindman} was combinatorial in nature. But later, using the algebraic structure of $\beta\mathbb{N}$, a very elegant proof of  Hindman's Theorem was established by Galvin and Glazer, which they never published. A proof of the Theorem \ref{hindman}, that uses the algebraic structure of $\beta\mathbb{N}$ was first presented in \cite[Theorem 10.3]{refCW}. One can also see the proof in \cite[Corollary 5.10]{refHS}. 

Let us first give a brief description of the algebraic structure of $\beta S_d$ for a discrete semigroup $(S, \cdot)$. We take the points of $\beta S_d$ to be the ultrafilters on
$S$, identifying the principal ultrafilters with the points of $S$ and thus pretending that $S\subseteq\beta S_d$. Given a set $A\subseteq S$, let us define the subsets of $\beta S_d$ by the following formula:

$$\cl A = \overline{A}= \{p\in\beta S_d : A\in p\}.$$

\noindent Then the set $\{\cl A\subseteq\beta S_d:A\subseteq\ben\}$ forms a basis for the closed sets of $\beta S_d$ as well as for the open sets. The operation $``\cdot"$ on $S$ can be extended to the Stone-\v{C}ech
compactification $\beta S_d$ of $S$, so that $(\beta S_d,\cdot)$ becomes a compact right topological semigroup (meaning that
for any $p \in \beta S_d$, the function $\rho_p : \beta S_d \rightarrow \beta S_d$, defined by $\rho_p(q) = q \cdot p$, is continuous)
with $S$ contained in its topological center (meaning that for any $x \in S$, the function
$\lambda_x : \beta S_d \rightarrow \beta S_d$, defined by $\lambda_x(q) = x \cdot q$, is continuous). A nonempty subset $I$ of a semigroup $T$ is called a \emph{left ideal of $S$} if $TI\subset I$,  a \emph{right ideal} if $IT\subset I$,
and a \emph{two-sided ideal} (or simply an \emph{ideal}\/) if it is both a left and a right ideal.
A \emph{minimal left ideal} is a left ideal that does not contain any proper left ideal.
Similarly, we can define a \emph{minimal right ideal} and the \emph{smallest ideal}.

Numakura proved in \cite{refN} (as remarked in \cite[Lemma 8.4]{refF}), that any compact Hausdorff right topological semigroup $T$ contains idempotents and therefore
has a smallest two-sided ideal

$$\begin{array}{ccc}
    K(T) & = & \bigcup\{L:L \text{ is a minimal left ideal of } T\} \\
         & = & \, \, \, \, \, \bigcup\{R:R \text{ is a minimal right ideal of } T\}.\\
  \end{array}$$

Given a minimal left ideal $L$ and a minimal right ideal
$R$, it easily follows that $L\cap R$ is a group and thus in particular contains
an idempotent.    If $p$ and $q$ are idempotents in $T$,
we write $p\leq q$ if and only if $p\cdot q=q\cdot p=p$. An idempotent
is minimal with respect to
this relation if and only if it is a member of the smallest ideal $K(T)$ of $T$.

Given $p,q\in\beta S$ and $A\subseteq S$, the set $A\in p\cdot q$ if and only if  $\{x\in S:x^{-1}A\in q\}\in p$, where $x^{-1}A=\{y\in S:x\cdot y\in A\}$. See \cite{refHS} for an elementary
introduction to the algebra of $\beta S$ and for any unfamiliar details.

A set $A\subseteq\mathbb{N}$ is called an IP$^{\star}$ set if it belongs to every
idempotent in $\beta\mathbb{N}$. Given a sequence $\langle
x_{n}\rangle_{n=1}^{\infty}$ in $\mathbb{N}$, we let $FP(\langle x_{n}\rangle_{n=1}^{\infty})$ be the product
analogue of finite sums. Given a sequence $\langle
x_{n}\rangle_{n=1}^{\infty}$ in $\mathbb{N}$, we say that $\langle
y_n\rangle_{n=1}^{\infty}$ is a {\it sum subsystem\/} of $\langle
x_{n}\rangle_{n=1}^{\infty}$, provided there is a sequence $\langle
H_{n}\rangle_{n=1}^{\infty}$ of nonempty finite subsets of
$\mathbb{N}$, such that $\max H_n<\min H_{n+1}$ and $y_n=\sum_{t\in
H_n}\,x_t$ for each $n\in\mathbb{N}$. The following Theorem  \cite[Theorem 2.6]{refBHic} shows that IP$^{\star}$ sets have substantially rich multiplicative structures.

\begin{theorem}\label{ip*}
Let $\langle x_{n}\rangle_{n=1}^{\infty}$ be a sequence in $\mathbb{N}$ and $A$ be an IP${^\star}$ set
in $(\mathbb{N},+)$. Then there exists a sum subsystem $\langle y_{n}\rangle_{n=1}^{\infty}$  of $\langle
x_{n}\rangle_{n=1}^{\infty}$ such that
 $$FS(\langle
y_n\rangle_{n=1}^\infty)\cup FP(\langle
y_n\rangle_{n=1}^\infty)\subseteq A.$$
\end{theorem}

Let us recall the definition of central set  \cite[Definition 4.42]{refHS}.

\begin{definition}
Let $S$ be a semigroup and $C\subseteq S$. Then $C$ is said to be {\it central} if
and only if there is some idempotent $p\in K(\beta S)$ such that $C\in p$.
\end{definition}

The algebraic structure of the smallest ideal of $\beta S$ plays a significant role in Ramsey
Theory. It is known that any central subset of $(\mathbb{N}, +)$ is guaranteed to have
substantial amount of additive combinatorial  structures. But Theorem 16.27 of \cite{refHS} shows that central sets in $(\mathbb{N},+)$
need not admit any multiplicative structure at all. On the other hand, in \cite[Theorem 2.4]{refBHic} we see that sets which belong to every minimal idempotent of $\beta\ben$, called central${^\star}$ sets, must have significant
multiplicative structures. In fact, central${^\star}$ sets in any semigroup $(S,\cdot)$ are defined to be those
sets which meet every central set.

\begin{theorem}\label{mcentral}
If $A$ is a central${^\star}$ set in $(\mathbb{N}, +)$, then it is also central in $(\mathbb{N},\cdot)$.
\end{theorem}

In case of central${^\star}$ sets, a result similar  to \ref{ip*} was proved in ~\cite[Theorem 2.4]{refDe} for a restricted class of sequences called minimal sequences. Recall that a sequence $\langle x_{n}\rangle_{n=1}^{\infty}$ in $\mathbb{N}$ is said to be a minimal sequence if
$$\bigcap_{m=1}^{\infty}\cl{FS(\langle
x_n\rangle_{n=m}^\infty)}\cap K(\beta \mathbb{N})\neq\emptyset.$$

\begin{theorem}\label{central*}
Let $\langle y_{n}\rangle_{n=1}^{\infty}$ be a minimal
sequence and let $A$ be a central${^\star}$ set in $(\mathbb{N},+)$. Then
there exists a sum subsystem $\langle x_{n}\rangle_{n=1}^{\infty}$
of $\langle y_{n}\rangle_{n=1}^{\infty}$ such that \break
$$FS(\langle x_{n}\rangle_{n=1}^{\infty})\cup
FP(\langle x_{n}\rangle_{n=1}^{\infty})\subseteq A.$$
\end{theorem}

A similar result was proved in a different setup in \cite{refDP}.

The original Central Sets Theorem was proved by Furstenberg in \cite[Theorem 8.1]{refF} (using a different
but equivalent definition of central sets). However, the most general version of Central Sets Theorem is in \cite{refDHS}. We state it here only for the case of a commutative semigroup.

\begin{theorem}
Let $(S,\cdot)$ be a commutative semigroup,
and let ${\cal T}={}^{\hbox {$\ben$}}\!S$ be the set of all
sequences in $S$. Let $C$ be a central subset of $S$. Then there
exist functions $\alpha:\pf({\cal T})\to S$ and
$H:\pf({\cal T})\to\pf(\ben)$ such that
\begin{enumerate}
\item if $F,G\in\pf({\cal T})$ and $F\subsetneq G$
then $\max H(F)<\min H(G)$, and
\item whenever $m\in\ben$, $G_1,G_2,\ldots,G_m\in
\pf({\cal T})$,
 $G_1\subsetneq G_2\subsetneq \ldots
\subsetneq G_m$, and for each $i\in\nhat{m}$,
$f_i\in G_i$, one
has
$$\prod_{i=1}^m \big(\alpha(G_i)\cdot\prod_{t\in H(G_i)}\,f_i(t)\big)\in C.$$
\end{enumerate}
\end{theorem}

Recently, a lot of attention has been paid to those sets which
satisfy the conclusion of the latest Central Sets Theorem.

\begin{definition}
Let $(S,\cdot)$ be a commutative semigroup,
and let ${\cal T}={}^{\hbox {$\ben$}}\!S$ be the set of all
sequences in $S$. A subset $C$ of $S$ is said to be a C-set if  there
exist functions $\alpha:\pf({\cal T})\to S$ and
$H:\pf({\cal T})\to\pf(\ben)$ such that
\begin{enumerate}
\item if $F,G\in\pf({\cal T})$ and $F\subsetneq G$,
then $\max H(F)<\min H(G)$ and
\item whenever $m\in\ben$, $G_1,G_2,\ldots,G_m\in
\pf({\cal T})$,
 $G_1\subsetneq G_2\subsetneq \ldots
\subsetneq G_m$, and for each $i\in\nhat{m}$,
$f_i\in G_i$, one
has
$$\prod_{i=1}^m \big(\alpha(G_i)\cdot\prod_{t\in H(G_i)}\,f_i(t)\big)\in C.$$
\end{enumerate}
\end{definition}

We now present some notations from \cite{refDHS}.

\begin{definition}
Let $(S,\cdot)$ be a commutative semigroup,
and let ${\cal T}={}^{\hbox {$\ben$}}\!S$ be the set of all
sequences in $S$.
\begin{enumerate}
\item A subset $A$ of $S$ is said to be a $J$-set if for every  $F\in\pf({\cal T})$
 there exist $a\in S$ and $H\in\pf(\ben)$ such that for all $f\in F$,\\
  $a\cdot\prod_{t\in H}f(t)\in A$.

 \item $J(S)$ $=$ $\{p\in\beta S$ $:$ $(\forall A\in p)(A$ is a $J$-set$)\}$.

 \end{enumerate}

\end{definition}
\begin{theorem}\label{J}
Let $(S, \cdot)$ be a discrete commutative semigroup and $A$ be a subset of $S$. Then  $A$
is a $J$-set if and only if $J(S) \cap\cl A\neq\emp$.
\end{theorem}

\begin{proof}
Since the collection of $J$-sets forms a partition regular family, the theorem follows from \cite[Therem 3.11]{refHS}.
\end{proof}
The following is a consequence of \cite[Theorem 3.8]{refDHS}. The easy proof for the commutative case can be found in \cite[Theorem 2.5]{refHS09}.

\begin{theorem}
Let $(S,+)$ be a commutative semigroup
and let ${\cal T}={}^{\hbox {$\ben$}}\!S$ be the set of all
sequences in $S$, and let $A\subseteq S$. Then $A$ is a $C$-set if and only if there is an idempotent $p\in \cl{A}\cap J(S)$.
\end{theorem}

We conclude these introductory discussions with the following \cite[Theorem 3.5]{refDHS}.

\begin{theorem}
If $(S,\cdot)$ be a discrete commutative semigroup then $J(S)$ is a closed two-sided ideal of $\beta S$ and $\cl K(\beta S)\subset J(S)$.

\end{theorem}

\section{$C^{\star}$-set}

We have already discussed IP$^{\star}$ and central$^{\star}$ sets before; let us now introduce the notion of $C^{\star}$-set.

\begin{definition}
Let $(S,\cdot)$ be a discrete commutative semigroup. A set $A\subseteq S$ is said to be a C$^{\star}$-set if it is a member of all the idempotents of $J(S)$.
\end{definition}

It is clear from the definition of C$^{\star}$-set that
\begin{center}
 IP$^{\star}$-set $\Rightarrow$ C$^{\star}$-set $\Rightarrow$ central$^{\star}$-set.
\end{center}

\begin{remark}
It is shown in \cite{refH09} that there exists a set $A\subset\ben$ which is a $C$-set in $(\ben,\, +)$, but its upper Banach density (defined below) vanishes. Since $A$ is a $C$-set in $(\ben,\, +)$, there exists an idempotent $p\in J(\ben)$ such that $A\in p$. But, as the upper Banach density of $A$ is zero, it is not a central set in $(\ben,\, +)$,  so  it is not contained in any minimal idempotent of $\beta\ben$. Hence $\ben\setminus A$ is a member of all the minimal idempotents in $\beta\ben$. Therefore $\ben\setminus A$ is a central$^{\star}$-set but not a C$^{\star}$-set as $\ben\setminus A\not\in p$.
\end{remark}

\begin{definition}
Let $(S,+)$ be a discrete commutative semigroup. A sequence $\langle
x_{n}\rangle_{n=1}^{\infty}$ is said to be an \emph{almost minimal sequence} if
$$\bigcap_{m=1}^{\infty}\cl{FS(\langle x_n\rangle_{n=m}^\infty)}\cap J(S))\neq\emptyset.$$
\end{definition}

To provide an example of an \emph{almost minimal sequence} which is not minimal, let us recall the following notion of ``Banach density''. The author is grateful to Prof. Neil Hindman for his help in constructing this example.

\begin{definition}Let $A\subset\ben$. Then
\begin{enumerate}
\item $d^{*}(A)$ $=$ sup$\{\alpha\in\ber:(\forall\,k\in\ben)(\exists n>k)(\exists a\in\ben)(|A\cap\{a+1,a+2,\ldots,a+n\}|\geq \alpha\cdot n)\}$.
\item $\triangle^{*}$ $=$ $\{p\in\beta\ben : (\forall A\in p)(d^{*}(A)>0)\}$.

\end{enumerate}
$d^{*}(A)$ is said to be the upper Banach density of $A$.

\end{definition}

It follows from \cite[Theorem 20.5 and 20.6]{refHS} that $\triangle^{*}$ is a closed two-sided ideal of $(\beta\ben,+)$, so that $\cl K(\beta\ben)\subset\triangle^{*}$.\\

 Let us now recall the following Theorem from \cite{refA}, which we shall require to construct example of an almost minimal sequence which is not minimal. 

\begin{theorem}\label{Adams1}
Let $\langle
x_{n}\rangle_{n=1}^{\infty}$ be a sequence in $\ben$, such that for all $n\in\ben$ we have $x_{n+1}>\sum_{t=1}^{n}x_t$, and $T=\bigcap_{m=1}^{\infty} \cl FS(\langle x_{n}\rangle_{n=m}^{\infty})$. Then the following conditions are
equivalent:
\begin{enumerate}
\item $T\cap K(\beta\ben)\neq\emp$;
\item $T\cap \cl K(\beta\ben)\neq\emp$;
\item the set $\{x_{n+1}-\sum_{t=1}^{n}x_t:n\in\ben\}$ is bounded.
\end{enumerate}
\end{theorem}
\begin{proof}
See \cite[Theorem 2.1]{refA}.
\end{proof}

We already know that a subset $A$ of $\ben$ is central if it belongs to  some idempotent of $K(\beta\ben)$. Further, the members of the ultrafilters of $K(\beta\ben)$ are piecewise syndetic.
Replacing piecewise syndeticity with positive upper Banach density leads to
the class of \emph{essential idempotents}: viz. $q \in\beta\ben$ is an \emph{essential idempotent} if it is an idempotent ultrafilter, all of whose elements have positive upper Banach
density-that is, $q\in\triangle^{*}$. By \cite[Theorem 2.8]{refBD}, a set $S\subseteq\ben$ is a $D$-set if it is contained in some essential idempotent.
The authors proved in \cite[Theorem 11]{refBBDF} that $D$-sets also satisfy the conclusion of the original Central Sets Theorem and are in particular $J$-sets.

Now let $d\in\ben$, and let us define a sequence $\langle x_{n}\rangle_{n=1}^{\infty}$  in $\ben$ by the following formula: $x_{n+1}=\sum_{t=1}^{n}x_t+\lfloor\frac{n+(d-1)}{d}\rfloor$. Then the set $\{x_{n+1}-\sum_{t=1}^{n}x_t:n\in\ben\}$ being unbounded, we have $\bigcap_{m=1}^{\infty} \cl FS(\langle x_{n}\rangle_{n=m}^{\infty})\cap K(\beta\ben)=\emp$. Therefore, the sequence $\langle x_{n}\rangle_{n=1}^{\infty}$ is not minimal. But by \cite[Lemma 2.20]{refA}, we have $\bigcap_{m=1}^{\infty} \cl FS(\langle x_{n}\rangle_{n=m}^{\infty})\cap \triangle^{*}\neq\emp$. Since $\bigcap_{m=1}^{\infty} \cl FS(\langle x_{n}\rangle_{n=m}^{\infty})\cap \triangle^{*}$ is a compact subsemigroup of $\beta\ben$, we can choose an idempotent $p$ in $\bigcap_{m=1}^{\infty} \cl FS(\langle x_{n}\rangle_{n=m}^{\infty})\cap \triangle^{*}$. In particular, $FS(\langle x_{n}\rangle_{n=1}^{\infty})\in p$. Therefore, by the above discussion, $FS(\langle x_{n}\rangle_{n=1}^{\infty})$ satisfies the conclusion of the original Central Sets Theorem and is in particular a $J$-set.  Hence by the foregoing Theorem \ref{minimalch}, we have $\bigcap_{m=1}^{\infty}\cl{FS(\langle
x_n\rangle_{n=m}^\infty)}\cap J(\ben))\neq\emptyset$. This implies that this sequence $\langle x_{n}\rangle_{n=1}^{\infty}$ is \emph{almost minimal}.

\begin{qs}
Does there exist a sequence $\langle x_{n}\rangle_{n=1}^{\infty}$ in $\ben$ with the property that $FS(\langle x_{n}\rangle_{n=1}^{\infty})$ is a $C$-set but its Banach density is zero?
\end{qs}

The author is thankful to the anonymous referee for her/his help in simplifying the following proof.

\begin{theorem}\label{minimalch}
In the  semigroup $(\ben,+)$, the following conditions are
equivalent:
\begin{enumerate}
\item[(a)]$\langle
x_{n}\rangle_{n=1}^{\infty}$ is an almost minimal sequence;
\item[(b)]$FS(\langle x_{n}\rangle_{n=1}^{\infty})$ is a $J$-set;
\item[(c)] there is an idempotent in  $\bigcap_{m=1}^{\infty}\cl{FS(\langle
x_n\rangle_{n=m}^\infty)}\cap J(\ben))\neq\emptyset$.
\end{enumerate}
\end{theorem}

\begin{proof}
$(a)\Rightarrow (b)$: The proof  follows from the definition of an almost minimal sequences.

$(b)\Rightarrow (c)$:  Let $FS(\langle x_{n}\rangle_{n=1}^{\infty})$ be a $J$-set. Then  by Theorem \ref{J},
we have $J(\ben)\cap \cl{FS(\langle x_{n}\rangle_{n=1}^{\infty})}\neq
\emptyset$. We choose $p\in J(\ben)\cap \cl{FS(\langle x_{n}\rangle_{n=1}^{\infty})}$. By \cite[Lemma 5.11]{refHS},
$\bigcap_{m=1}^{\infty}\cl{FS(\langle x_{n}\rangle_{n=m}^{\infty})}$ is a subsemigroup of
$\beta \ben$. As a consequence of \cite[Theorem 3.5]{refDHS}, it follows that $J(\ben)$ is a subsemigroup of $\beta\ben$.  Also, $J(\ben)$ being closed, it is a compact subsemigroup of $\beta\ben$. Therefore, it suffices to show that for each $m$, $\bigcap_{m=1}^{\infty}\cl{FS(\langle
x_{n}\rangle_{n=m}^{\infty})}\cap J(\ben) \neq \emptyset$, because then it must contain an idempotent.
For this, it in turn suffices to show  that, for any $m\in\ben$ we have  $\cl{FS(\langle
x_{n}\rangle_{n=m}^{\infty})}\cap J(\ben) \neq \emptyset$. So  let $m\in\mathbb{N}$ be given. Then \\
$FS(\langle
x_{n}\rangle_{n=1}^{\infty})$\\
$=FS(\langle
x_{n}\rangle_{n=m}^{\infty})\cup FS(\langle
x_{n}\rangle_{n=1}^{m-1})\cup\bigcup\{t+FS(\langle
x_{n}\rangle_{n=m}^{\infty}):t\in FS(\langle
x_{n}\rangle_{n=1}^{m-1})\}$.\\
 Hence we must have one of the followings:
\begin{enumerate}

\item $FS(\langle x_{n}\rangle_{n=1}^{m-1})\in p$;

\item $FS(\langle x_{n}\rangle_{n=m}^{\infty})\in p$;

\item $t+FS(\langle x_{n}\rangle_{n=m}^{\infty})\in p$ for some
$t\in FS(\langle x_{n}\rangle_{n=1}^{m-1})$.
\end{enumerate}

Clearly (1) does not hold, because in that case $p$ becomes a member of  ${\mathbb N}$, that is a principal ultrafilter, while  $p\in\beta{\mathbb N}\setminus{\mathbb N}$.  If (2) holds, then we are done.  So assume that (3) holds. Then for some $t\in
FS(\langle x_{n}\rangle_{n=1}^{m-1})$, we have that  $t+FS(\langle x_{n}\rangle_{n=m}^{\infty})\in p$. We choose some $q\in\cl{FS(\langle x_{n}\rangle_{n=m}^{\infty})}$ so that $t+q=p$. For every $F\in q$, we have $t\in$ $\{n\in\ben$ $:\,-n+(t+F)\in q\}$ so that $t+F\in p$. Since $J$-sets in $(\ben\,,+)$ are translation invariant, $F$ becomes a $J$-set. Thus
$q\in J(\ben)\cap \cl{FS(\langle x_{n}\rangle_{n=m}^{\infty})}$.\\

$(c)\Rightarrow (a)$: This case is obvious.
\end{proof}

\begin{lemma}
If $A$ is a $C$-set in $(\ben, +)$ then for any $n\in\ben$, $nA$ and $n^{-1}A$ are also C-sets,  where $n^{-1}A=\{m\in\ben:n\cdot m\in A\}$.

\end{lemma}
\begin{proof}
\cite[Lemma 8.1]{refLi}.
\end{proof}

\begin{lemma}\label{C*} If $A$ is a $C^{\star}$-set in $(\ben,\, +)$ then $n^{-1}A$ is also a $C^{\star}$-set for any $n\in\ben$.

\end{lemma}

\begin{proof}
Let $A$ be a $C^{\star}$-set and $t\in\ben$. To prove that $t^{-1}A$ is a $C^{\star}$-set, it is sufficient to show that for any C-set $C$, we have $C\cap t^{-1}A\neq \emp$. Since $C$ is a C-set, $tC$ is also a C-set, so that $A\cap tC\neq \emp$. Choose $n\in tC\cap A$ and $k\in C$ such that $n = tk$. Therefore $k=n/t\in t^{-1}A$. Hence $C\cap t^{-1}A\neq \emp$.
\end{proof}

\begin{theorem}\label{C*comb}
 Let $\langle x_{n}\rangle_{n=1}^{\infty}$ be an almost minimal
sequence and $A$ be a $C^{\star}$-set in $(\ben,\, +)$. Then
there exists a sum subsystem $\langle y_{n}\rangle_{n=1}^{\infty}$
of $\langle x_{n}\rangle_{n=1}^{\infty}$ such that
$$FS(\langle y_{n}\rangle_{n=1}^{\infty})\cup
FP(\langle y_{n}\rangle_{n=1}^{\infty})\subseteq A.$$
\end{theorem}

\begin{proof}
Since $\langle x_{n}\rangle_{n=1}^{\infty}$ is an almost minimal
sequence, by Theorem \ref{minimalch}, we can find some idempotent $p\in J(\ben)$ such that $FS(\langle x_{n}\rangle_{n=m}^{\infty})\in p$, for each $m\in\ben$.  Again, since $A$ is a $C^{\star}$-set in $(\ben,\, +)$, from Lemma \ref{C*}, it follows that, $s^{-1}A\in p$, for every $s\in \ben$. Let $A^{\star}=\{s\in A$ : $-s+A\in p\}$. Then by ~\cite[Lemma 4.14]{refHS}, we have $A^{\star}\in p$. Then we can choose $y_{1}\in A^{\star}\cap FS(\langle x_{n}\rangle_{n=1}^{\infty})$. Inductively, let $m\in \mathbb{N}$ and $\langle y_{i}\rangle_{i=1}^{m}$, $\langle H_{i}\rangle_{i=1}^{m}$ in $\mathcal{P}_{f}(\mathbb{N})$ be chosen with the following properties:

\begin{enumerate}
\item for all $i\in \{1,2,\ldots,m-1\}$, $\max H_{i}< \min H_{i+1}$;
\item if $y_{i}=\sum_{t\in H_{i}}x_{t}$ then $\sum_{t\in H_{m}}x_{t}\in A^{\star}$ and $FP(\langle y_{i} \rangle_{i=1}^{m})\subseteq A^{\star}$.
\end{enumerate}

We observe that $\{\sum_{t\in H}x_{t}$ : $H\in \mathcal{P}_{f}(\mathbb{N}), \min H> \max H_{m}\}\in p$.
Let us set $B=\{\sum_{t\in H}x_{t}$ : $H\in \mathcal{P}_{f}(\mathbb{N}),\,\, \min H> \max H_{m}\}$,  $E_{1}=FS(\langle y_{i}\rangle_{i=1}^{m})$ and $E_{2}=FP(\langle y_{i}\rangle_{i=1}^{m})$. Now consider
$$D=B\cap A^{\star}\cap \bigcap _{s\in E_{1}}(-s+A^{\star})\cap \bigcap _{s\in E_{2}}(s^{-1}A^{\star}).$$
Then $D\in p$. Choose $y_{m+1}\in D$ and $H_{m+1}\in \mathcal{P}_{f}(\mathbb{N})$ such that $\min H_{m+1}> \max H_{m}$. Putting  $y_{m+1}=\sum_{t\in H_{m+1}}x_{t}$, it shows that the induction can be continued and this eventually proves the theorem.
\end{proof}

\noindent\textbf{Acknowledgements.}
The author is thankful to Prof. Neil Hindman for his helpful suggestions and comments. The author would also like to thank the anonymous referee for many helpful suggestions and corrections, and in particular for the assistance she/he provided towards simplifying the proof of Theorem \ref{minimalch}.

\bibliographystyle{amsplain}

\end{document}